\documentclass[twoside,11pt]{article}
\newtheorem{theorem}{Theorem}[section]
\newtheorem{algorithm}{Algorithm}[section]
\newtheorem{remark}{Remark}[section]

\newenvironment{proof}[1]{\noindent{\em Proof{#1}:}}{\hfill $\square$ \\}

\usepackage{epsf,amssymb,amssymb,graphicx,natbib}

\begin{document}
\pagestyle{myheadings}

\title{{\bf Resource Allocation Based on Past Incident Patterns}}
\author{{\bf M.N.M.\ van Lieshout}\\[0.1in]
Centrum Wiskunde \& Informatica \\
P.O.\ Box 94079, NL-1090 GB, Amsterdam, The Netherlands\\[0.1in]
Department of Applied Mathematics, University of Twente \\
P.O.\ Box 217, NL-7500 AE, Enschede, The Netherlands}
\date{}

\maketitle

\begin{center}
    {\bf Dedicated to Professor J.\ Glaz}
\end{center}

\noindent
{\bf Abstract:} 
We formulate and solve two resource allocation problems
motivated by a preparedness question of emergency response services. 
First, we consider the assignment of vehicles to stations, and, in a 
second step, assign crews to vehicles. In both cases, we work in a 
minimax framework and define the objective function for a spatial 
catchment area as the total risk in this area per resource unit 
allocated to it. The solutions are explicit and can be calculated in 
practice by a greedy algorithm that successively allocates a resource 
unit to an area having maximal relative risk, with suitable tie breaker 
rules. The approach is illustrated on a data set of incidents reported
to the Twente Fire Brigade.  \\[0.1in]

\noindent
{\em AMS Mathematics Subject Classification (2020 Revision):}
60G55, 90C10, 91B32.
\\
\noindent
{\em Key words \& Phrases:} 
Capacity planning, discrete optimisation, intensity function, point process.

\section{Introduction}

Consider the following motivating capacity planning problem
encountered by emergency response services such as the ambulance 
or fire and rescue brigade. In the Netherlands, these services are 
organised regionally. In each region, there are a number of stations, 
each equipped with at least one resource, for example, an ambulance, a
fire truck or a crew. Note that the regions are heterogeneous in terms of 
risk. For example, densely populated areas may need more ambulances, 
whereas woodland regions are more susceptible to wildfires. Thus,
the goal of capacity planners is to distribute the available resources 
as well as possible over the stations taking into account the total
risk in the stations' catchment areas.  

Solving problems such as those outlined above relies on two key tools: 
a reliable estimator of the risk based on historic incident occurrence
data and practical algorithms for solving integer-valued optimisation 
problems.

Non-parametric estimation of the expected number of incidents
in a given region and time window is an old problem. At the most elementary
level, the raw counts provide an unbiased estimator \citep{DuRi29}. 
However, the resulting risk map depends heavily on the region 
boundaries and, since it lacks smoothness, is not visually pleasing.
The first problem can be overcome by using the Voronoi or Delaunay 
cells of the observed incidents 
\citep{BarrScho10,Lies12,Moraetal19,Ord78,SchaWeyg00} but without 
further smoothing the resulting map remains very spiky. Thus, 
the most widely used approach remains kernel smoothing \citep{Digg85}. 
The amount of smoothing is governed by a bandwidth parameter, the 
selection of which has recently attracted quite some attention 
\citep{CronLies18,Lies20,Lo17}. Allowing the bandwidth to vary over 
space is advantageous when there are large regional differences 
\citep{Abra82,Davietal18,Lies21}. When covariate information is
available, kernel smoothing can be performed in the covariate domain
\citep{Gua08}, at least when the dimension is moderate. For larger
dimensions, machine learning approaches may give very accurate
estimators \citep{BiscLava25,Kimetal22,Luetal25} albeit due to 
the many tuning parameters that need to be set, e.g., using 
cross-validation \citep{Cronetal23}, the computational cost may be high.

\citet{Koop53} seems to have been the first to formalise a resource
allocation problem. In its generic form, the goal is to minimise the 
sum of a convex cost function on the non-negative integers over a finite
number of activities when the number of available resources is bounded.
This problem can be solved by a greedy algorithm \citep{Gros56}. 
In computational terms, the algorithm is not optimal. Hence, 
some effort has been spent on finding faster, polynomial-time
ones \citep{FredJohn82,Grotetal93}. Over the years, various 
generalisations have been considered. For instance, the assumption 
that the objective function is the sum of marginal contributions may 
not be appropriate when there is interaction between the activities. 
When a more even distribution of resources is desirable, working in 
a minimax framework might be more appropriate \citep{KleiLuss91,Luss92}. 
Furthermore, the convexity assumption may be relaxed \citep{Muro98}. 
In a different direction, there may be additional constraints, e.g., 
because different activities require different resources, an activity 
may need more than one resource, or, reversely, cannot take more than
a given number of resources \citep{Fuji05}. For a thorough discussion, 
see \citet{Katoetal25} or \citet{Schr03}.

The plan of this paper is as follows. In Section~\ref{S:intensity},
we fix notation and discuss how to estimate the intensity function 
of the incident process. Section~\ref{S:allocation} discusses the 
state of the art in integer programming with particular attention
to resource allocation. In Section~\ref{S:capacity}, we formulate and
solve a capacity planning problem for vehicles. Having found an
optimal vehicle allocation, Section~\ref{S:crew} turns to the 
assignment of crews to vehicles. To demonstrate the efficacy of 
the algorithms in practice, we illustrate the approach on a data 
set supplied by the Twente Fire Brigade in Section~\ref{S:Twente}.
The paper closes with a discussion and suggestions for future research.

\section{Estimating the incident intensity function}
\label{S:intensity}

Let $\mathcal{S} \subset \mathbb{R}^2$ be a finite set. In a security
context, an element $s\in \mathcal{S}$ could be the location of 
a hospital from which ambulances are dispatched or the position of 
a fire station. To estimate the risk in the catchment areas $C(s)$
of $s$, assume that we have at our disposal a realisation of the 
simple point process $\Psi$ of incidents 
in a non-empty, bounded and open window $W \subset \mathbb{R}^2$ in a 
fixed time interval. Then, for any Borel set $B\subset W$, the first moment 
measure $\Lambda(B)$ is defined as the expected number of incidents that 
occur in $B$. We will assume that there exists a non-negative integrable 
function $\lambda: \mathbb{R}^2 \to \mathbb{R}^+$ such that 
\[
\Lambda(B) = \int_B \lambda(u) du.
\]
The function $\lambda$ is called the intensity function of $\Psi$.
Hence, to estimate $\Lambda( C(s) )$, it suffices to estimate $\lambda$.
For measure-theoretic details, we refer the reader to \citet{SKM}.

To estimate $\lambda$, since the catchment areas of emergency response 
services may be quite heterogeneous, we use an adaptive kernel estimator 
\citep{Abra82,Davietal18,Lies21} defined for $x_0 \in W$ as
\begin{equation}
\label{e:Abramson}
\widehat\lambda( x_0; h, \Psi, W )  =
  \sum_{y\in \Psi\cap W}  \frac{1}{c(y)^d h^d} \kappa\left(
    \frac{x_0-u}{h c(u) } \right) w(u, h, W)^{-1},
\end{equation}
where
\begin{equation}
\label{e:Abra-c}
c(y) = \left(
  \frac{\lambda(y)}{ \prod_{z\in\Psi\cap W} \lambda(z)^{1/N(\Psi\cap W)} }
\right)^{-1/2},
\end{equation}
$N(\Psi\cap W)$ denotes the number of points of $\Psi$ that fall in $W$,
$\kappa$ is the Gaussian kernel, $h>0$ is the bandwidth
and $w(y,h, W)$ is an edge correction weight. When it is undesirable
that mass of $\kappa$ leaks away to the complement of $W$, global
edge correction \citep{Lies12} is appropriate:
\[
w(u,h,W) =  \int_W \frac{1}{c(u)^d h^d} \kappa\left(
    \frac{v-u}{h c(u) } \right) dv.
\]
Note that the local bandwidth $h c(u)$ is larger for $u$ in 
regions with a low incident rate, thus smoothing more when few
data points are available. Also observe that, as it depends on 
the unknown true intensity function $\lambda$, (\ref{e:Abramson}) 
cannot be calculated. A common solution is to estimate $c(u)$ by 
plugging in a so-called pilot intensity estimator, that is, a 
rough guess of $\lambda$. For example, one could estimate
$\lambda(y)$ by a global bandwidth kernel estimator (taking
$c(y)\equiv 1$ in expression (\ref{e:Abramson})).

To select the bandwidths, two main approaches are available: 
cross-vali\-dation \citep{Load99} and the CvL method 
\citep{CronLies18,Lies24}. The former method maximises the 
leave-one-out cross-vali\-dation log likelihood
\begin{equation}
\label{e:bwPPL}
L(h;\Psi, W)=
\sum_{x\in\Psi\cap W} 
  \log \widehat \lambda(x; h, \Psi\setminus \{ x \}, W)
- \int_W \widehat \lambda( u; h, \Psi, W) \, du
\end{equation}
over $h$. For the CvL method, estimate the area $\ell(W)$
of $W$ by
\[
\label{e:HamFun}
T(h;\Psi, W) = \left\{ \begin{array}{ll}
  \displaystyle
  \sum_{x\in\Psi\cap W} \frac{1}{\widehat\lambda(x;h, \Psi, W)},
& \quad \Psi\cap W \neq \emptyset, \\
\ell(W), & \quad \mbox{ otherwise},
\end{array} \right.
\]
and choose bandwidth $h>0$ by minimising the absolute error
\begin{equation}
\label{e:DefHam}
F(h;\Psi, W) =
\left| T(h;\Psi, W) -\ell(W) \right|.
\end{equation}
Which of the two techniques leads to smaller errors depends
on the data at hand. The cross-validation method is 
computationally more expensive and tends to select
smaller bandwidths than CvL; the latter is cheaper to calculate but
tends to over-smooth for repulsive patterns. 

\section{Resource allocation}
\label{S:allocation}

Write $K \in \mathbb{N}$ for the available number of resources and let
$f_s: \mathbb{N}_0 \to \mathbb{R}$ be a convex cost function for 
each $s\in \mathcal{S}$. The classic resource allocation problem aims 
to minimise 
\[
\sum_{s\in\mathcal{S}} f_s( n(s) )
\]
subject to 
\[
 \sum_{s\in\mathcal{S}} n(s) = K
\]
and, for all $s\in\mathcal{S}$, the number of resources assigned to 
$s$ takes non-negative integer values, i.e., $n(s)\in \mathbb{N}_0$.
For this problem, it is well-known that an optimal allocation can be 
calculated by the following greedy algorithm.
Start with $n(s) = 0$ and iteratively allocate one unit of the
resource to that activity that minimises the increase in the 
current objective value until $K$ resources have been allocated.
Formally, write, for $i=1, \dots, K$, 
\[
d_s(i) =  f_s(i) - f_s(i-1) 
\]
and let $V$ be the set of the $K$ smallest elements in 
$\{ d_s(i): s\in\mathcal{S}, i = 1, \dots, K \}$. In case of
ties,
if $d_s(i-1) = d_s(i)$,  preference is given to $d_s(i-1)$. Then 
\[
n^*(s) = \left\{ \begin{array}{ll}
0, & \mbox{ if } d_s(1) \not \in V, \\
K, & \mbox{ if } d_s(K) \in V, \\
i, & \mbox{ if } d_s(i) \in V \mbox{ and } d_s(i+1) \not \in V,
\end{array}
\right.
\]
is an optimal solution \cite[Theorem 1]{Katoetal25}. The approach
is efficient from a computational point of view when $K$ is moderate.
Note that a minimisation problem can be turned into a maximisation
problem by negating of functions. Although not immediately obvious,
also minimax problems can be transformed to fit the framework 
discussed above
by considering cumulative sums. Indeed, suppose that $f_s$ are 
non-decreasing non-negative functions and write
\[
g_s(i) = \sum_{j=0}^i f_s(j), \quad i = 0, 1, \dots.
\]
Then, by Theorem~9  in 
\citet{Katoetal25}, any optimal solution to
\[
\sum_{s\in \mathcal{S}} g_s(n(s))
\]
such that $\sum_s n(s) = K$ and $n(s) \in \mathbb{N}_0$,
$s\in \mathcal{S}$,
is also an optimal solution to the minimax problem
\[
\min\max_{s\in \mathcal{S}} f_s(n(s))
\]
subject to the same constraints.

\section{Allocation of vehicles}
\label{S:capacity}

We first consider the allocation of vehicles such as ambulances or
fire trucks to stations from which they are dispatched to incidents. We 
assume that each station $s \in \mathcal{S}$ serves a catchment area 
$C(s) \subset W$ and is equipped with at least one vehicle.
We suppose that the $C(s)$ are Borel sets and form a partition 
of $W$. Given an estimator $\widehat \lambda$ of the incidents, set
\[
\Lambda(s) = \int_{C(s)} \widehat \lambda(s) \, ds
\]
for the estimated total risk in $C(s)$. Write $n(s)$ for the number 
of vehicles operating from station $s$, and suppose that $K$ vehicles 
are available in total. Then the average risk per vehicle for station 
$s$ is ${\Lambda(s)}/{n(s)}$. As one would like to have a small average 
risk for all stations, a suitable problem formulation is as follows. Find 
\begin{equation}
\min\max_{s \in \mathcal{S}} \frac{\Lambda(s)}{n(s)}
\label{e:vehicles}
\end{equation}over all $n: \mathcal{S} \to \mathbb{N}$ such that 
$\sum_{s\in\mathcal{S}} n(s) = K$, where the constant
$K$ satisfies $K \geq |\mathcal{S}|$.  Consequently, for all $s$,
$1\leq n(s) \leq K$ so that every station is equipped with
at least one resource and no station can have more than the
available resources. In fact $n(s) \leq K - |S| + 1$, writing
$|\cdot|$ for cardinality.

Note that the function $\Lambda(s)/n(s)$ is decreasing in
$n(s)$ and that $n(s)$ cannot be zero. The following modification
of the greedy algorithm discussed in Section~\ref{S:allocation} 
can be used to solve~(\ref{e:vehicles}).

\begin{algorithm}
Initialise by allocating one resource unit to every station,
i.e. $n_{|\mathcal{S}|}(s) = 1$ for all $s\in \mathcal{S}$.
At iteration $k$, after $k$ resources have already been 
allocated according to $(n_k(s))_{s\in\mathcal{S}}$, 
add one to a station $s^*$ for which 
\[
s^* = \arg\max_{s\in \mathcal{S}} \frac{\Lambda(s)}{n_k(s)}
\]
to obtain 
\[
\left\{ 
\begin{array}{llll}
    n_{k+1}(s) & = & n_k(s) + 1, & \mbox{ if } s = s^* \\
    n_{k+1}(s) & = & n_k(s), & \mbox{ if } s\neq s^*. \\
\end{array} 
\right.
\]
\label{A:pumper}
\end{algorithm}

Algorithm~\ref{A:pumper} leads to the following 
optimal resource allocation $n(s) := n_K(s)$. Set, for 
$K \geq |\mathcal{S}|$,
\[
V = \left\{ \frac{\Lambda(s)}{1+k} : s \in \mathcal{S}, 
 k \in \{  0, \dots, K-|\mathcal{S}| \} \right\}.
\]
If $K = |\mathcal{S}|$, $n_K(s) = 1$ by definition. Otherwise,
for $K > |\mathcal{S}|$, order the members of $V$ and collect the 
$K-|\mathcal{S}|$ largest ones in $V_K$. Then, for $s\in\mathcal{S}$,
\begin{equation}
n(s) = \left\{ \begin{array}{ll}
    1 &  \mbox{if } \Lambda(s) \not \in V_K, \\
   1 + K-|\mathcal{S}|  & \mbox{if } 
   \frac{\Lambda(s)}{K-|\mathcal{S}|} \in V_K, \\
   k  & \mbox{if } \frac{\Lambda(s)}{k-1} \in V_K,
      \frac{\Lambda(s)}{k} \not \in V_K, \quad k=2, \dots, K-|\mathcal{S}|.
\end{array} \right.
\label{e:pumper}
\end{equation}

We are now ready to prove optimality.

\begin{theorem} \label{t:cars}
Suppose that $K \geq |\mathcal{S}|$ and that $\Lambda(s) > 0$
for all $s\in \mathcal{S}$. Then, Algorithm~\ref{A:pumper} with
$k=K$ returns a decision rule $n(s)$, $s\in \mathcal{S}$, that 
minimises the maximal average risk (\ref{e:vehicles}) and that 
satisfies the constraints  $1\leq n(s)$ for all $s$ and 
$\sum_{s\in\mathcal{S}} n(s) = K$.
\end{theorem}

\begin{proof}{}
As the greedy algorithm is initialised with $n_{|\mathcal{S}|}(s) 
= 1$ for all $s\in\mathcal{S}$ and continues until $K$ resources 
have been allocated, the constraints hold. 

The optimality proof proceeds by contradiction. If 
$(n_k(s))_{s\in\mathcal{S}}$ is not optimal, there exists an optimal 
solution $(n(s))_{s\in\mathcal{S}}$ that is strictly better, i.e.,
\[
    \max_{s\in\mathcal{S}} \frac{\Lambda(s)}{n(s)}
    < \max_{s\in\mathcal{S}} \frac{\Lambda(s)}{n_k(s)}.
\]
Without loss of generality, we may take that optimum 
$n(\cdot)$ for which
\[
\sum_{s\in\mathcal{S}} |n(s) - n_k(s)|
\]
is smallest. Since 
$\sum_s n(s) = \sum_s n_k(s) = K$, 
it follows that there exist $s$ and $s^\prime$ in 
$\mathcal{S}$ such that $n_k(s) > n(s)$ and 
$n_k(s^\prime) < n(s^\prime)$.  Now, construct a new solution 
$\tilde n$ from $n$ by adding a resource to $s$ and withdrawing 
one from $s^\prime$ and note that $\tilde n(\cdot)$ satisfies the 
constraints and has smaller $\ell_1$-distance to $n_k(\cdot)$. 
Returning to $n$, observe that
\[
\frac{\Lambda(s^\prime)}{n(s^\prime)-1} \leq 
\frac{\Lambda(s^\prime)}{n_k(s^\prime)}
\mbox{ and }
\frac{\Lambda(s)}{n_k(s)-1} \leq \frac{\Lambda(s)}{n(s)}.
\]
Since 
\[
\frac{\Lambda(s^\prime)}{n_k(s^\prime)} \leq 
\frac{\Lambda(s)}{n_k(s)-1},
\]
as  the greedy algorithm has selected $n_k(s)$ but not
$n_k(s^\prime) + 1 $,  we obtain
\begin{equation}
\frac{\Lambda(s^\prime)}{n(s^\prime)-1} \leq
\frac{\Lambda(s)}{n(s)}. 
\label{e:s-sprime}
\end{equation}
We claim that 
\begin{equation}
\max_{t\in\mathcal{S}} \frac{\Lambda(t)}{\tilde n(t)}
\leq \max_{t \in \mathcal{S} } \frac{\Lambda(t)}{n(t)}.
\label{e:maxT}
\end{equation}
By (\ref{e:s-sprime}), 
\[
\frac{\Lambda(s^\prime)}{\tilde n(s^\prime)}=
\frac{\Lambda(s^\prime)}{n(s^\prime)-1} \leq
\frac{\Lambda(s)}{n(s)} \leq \max_{t\in\mathcal{S}} \frac{\Lambda(t)}{n(t)}.
\]
Also
\[
\frac{\Lambda(s)}{\tilde n(s)} = \frac{\Lambda(s)}{n(s)+1} \leq
\frac{\Lambda(s)}{n(s)} \leq  \max_{t \in \mathcal{S}} \frac{\Lambda(t)}{n(t)}.
\]
Since for $t \not \in \{ s, s^\prime \}$, $\tilde n(t) = n(t)$,
claim (\ref{e:maxT}) holds. Thus, $\tilde n(\cdot)$ is at least as good 
as $n$ but with strictly smaller $\ell_1$-distance to $n_k$, which 
leads to a contradiction.
\end{proof}

\section{Personnel allocation}
\label{S:crew}

In the previous section, we discussed how to assign vehicles, 
e.g., fire trucks or ambulances, to regions. In practice, these vehicles 
require a crew of some fixed size  $\alpha  \in \{ 2, 3, \dots \}$. 
Assume that the stations $s \in \mathcal{S}$ have been equipped with 
$n(s)$ vehicles. In order to distribute personnel over stations,
write $f(s)$ for the number of people assigned to $s$. With these
staff, $\lfloor f(s)/\alpha \rfloor$ vehicles can be operated. 
Therefore, the effective risk per vehicle for station $s$ is 
\begin{equation}
\frac{\Lambda(s)}  { \lfloor f(s)/\alpha \rfloor }
\label{e:percrew}
\end{equation}
when the nearest integer $ \lfloor f(s)/\alpha \rfloor$ smaller 
than or equal to $f(s)/\alpha$ is less than or equal to $n(s)$. 
Otherwise, all vehicles are operational and the risk per vehicle 
is  ${\Lambda(s)} / n(s) $ as before. In the sequel, we shall work 
under the assumption of shortage, while maintaining the requirement 
that each catchment area $C(s)$ is served by at least one operational 
vehicle, that is, 
\[
 \alpha |\mathcal{S}| \leq \sum_{s\in \mathcal{S}} f(s) 
 = K < \alpha \sum_{s\in \mathcal{S}} n(s),
\]
where $K$ is the total amount of resources (crew members).

The considerations above lead to the following optimisation problem.
Find
\begin{equation}
\min\max_{s \in \mathcal{S}} \frac{\Lambda(s)}
{ \lfloor f(s)/\alpha \rfloor}
\label{e:humans}
\end{equation}
over all $f: \mathcal{S} \to \mathbb{N}$ such that, for all
$s \in \mathcal{S}$,
\[ 
\left\{ \begin{array}{l}
f(s) \geq \alpha,\\
\sum_{s\in\mathcal{S}} f(s) = K ,\\
f(s) \leq \alpha \, n(s).
\end{array} \right.
\]
The station-dependent last constraint is required to avoid 
assigning resources to stations without sufficient equipment.

Note that for each station there will be ties in the average risk 
function (\ref{e:percrew}) that have to be dealt with 
properly in a greedy allocation algorithm. The following modification of
Algorithm ~\ref{A:pumper} can be used to solve~(\ref{e:humans}).

\begin{algorithm}
Initialise by allocating $\alpha$ resource units to every station,
i.e. $f_{\alpha |\mathcal{S}|}(s) = \alpha$ for all $s\in \mathcal{S}$.
At iteration $k$, after $k$ resources have already been 
allocated according to $(f_k(s))_{s\in\mathcal{S}}$, 
add one to a station $s^*$ for which 
\[
s^* = \arg\max_{ f_k(s) < \alpha n(s)} 
  \frac{\Lambda(s)}{ \lfloor f_k(s)/ \alpha \rfloor}
\]
to obtain 
\[
\left\{ \begin{array}{llll}
    f_{k+1}(s) & = & f_k(s) + 1,  & \mbox{ if } s = s^*, \\
    f_{k+1}(s) & = & f_k(s), & \mbox{ if } s\neq s^*.
\end{array} \right.
\]
In case of ties, preference is given to those stations
$s$ for which $f_k(s) \, \rm{mod} \, \alpha$ is largest.
\label{A:men}
\end{algorithm}

A few remarks on the tie-break method are in order. 

\begin{remark}
{\rm Set $\alpha = 6$ and $K=18$ so that three vehicles can 
be taken into operation. Suppose that there are two stations, 
$s_1$ and $s_2$, each equipped with two vehicles ($n(s_1) 
= n(s_2) = 2$), whose catchments areas $C(s_1)$ and $C(s_2)$ 
are equally risky with $\Lambda(s_1) = \Lambda(s_2) = 6$.
In the first iteration of Algorithm~\ref{A:men} after 
initialising $f_{12}(s_1) = f_{12}(s_2) = 6$, 
one of the stations is allocated an extra resource, say $s_1$. 
Then, because of the preference condition, also in the next 
iterations, resources are allocated to $s_1$ so that 
at the end of the algorithm 
$\Lambda(s_1)/ \lfloor f_{18}(s_1)/6 \rfloor = 3$ and 
$\Lambda(s_2)/ \lfloor f_{18}(s_2)/6 \rfloor = 6$. Without
the preferential allocation to $s_1$, a resource
could also be allocated to $s_2$ leading to 
$\Lambda(s_1)/ \lfloor f_{18}(s_1) /6 \rfloor = 6$ and 
$\Lambda(s_2)/ \lfloor f_{18}(s_2) /6  \rfloor = 6$.
This allocation has the same maximal risk per crew, but 
does not use capacity efficiently.} $\hfill\Box$
\end{remark}

We proceed to show that Algorithm~\ref{A:men} solves the 
optimisation problem (\ref{e:humans}).

\begin{theorem} \label{t:personnel}
Suppose that $n(s) \geq 1$ and $\Lambda(s) > 0$ for all $s\in \mathcal{S}$.
Let $\alpha |\mathcal{S}| \leq K < \alpha \sum_{s\in \mathcal{S}} n(s)$.
Then Algorithm~\ref{A:men} with $k=K$ returns a decision rule $f(s)$,
$s\in \mathcal{S}$ that minimises the maximal average risk (\ref{e:humans}) 
and satisfies the constraints that $\alpha \leq f(s) \leq \alpha n(s)$ 
for all $s\in \mathcal{S}$ and that $\sum_{s\in \mathcal{S}} f(s) = K$.
\end{theorem}

\begin{proof}{}
The greedy algorithm is initialised with 
$f_{\alpha |\mathcal{S}}(s) = \alpha$ and 
continues until $K$ resources have been allocated. 
Additionally, in each step, the addition of one resource
is constrained to $f_k(s) < \alpha  n(s)$.
Hence  the constraints are satisfied. 

The optimality proof proceeds by contradiction.  If $(f_k(s))_{s\in\mathcal{S}}$
is not optimal, there exists an optimal solution $(f(s))_{s\in\mathcal{S}}$ that is strictly better, i.e.,
\[
    \max_{s\in\mathcal{S}} \frac{\Lambda(s)}
    { \lfloor f(s) / \alpha \rfloor }
    < \max_{s\in\mathcal{S}} \frac{\Lambda(s)}
     { \lfloor f_k(s) / \alpha \rfloor }.
\]

By the tie-breaker rule, the 
output of Algorithm~\ref{A:men}, $f_k$, contains at 
most one station with slack, which has size
$K - \alpha\lfloor K/\alpha \rfloor \leq \alpha -1$,
and for all other stations, $f_k$ is a multiple of $\alpha$. 

For $f$, if the combined slack is $\alpha$ or more, pick any 
$s_0 = \arg\max_s \Lambda(s) /$ $ \lfloor 
( f(s) / \alpha ) \rfloor$ with 
$\alpha \lfloor f(s_0) / \alpha \rfloor + \alpha \leq n(s_0)$.
If the maximiser $s_0$ is unique, 
allocating units out of the slack to station $s_0$ to 
fill up an additional vehicle 
would lead to a smaller maximal risk per crew, thus 
contradicting the assumption that $f$ is optimal. In case
of a tie, this action would leave the maximal risk per crew 
unchanged and we repeat this procedure
until the total slack is smaller than 
$\alpha$ or no maximiser $s_0$ can accommodate an extra crew.
In the latter case, filling up  remaining stations to
full capacity  as long as there at least $\alpha$ units
left, does not change the maximal risk per vehicle and
the procedure terminates because of the scarcity assumption
$K < \alpha  \sum_{s} n(s)$. 

In conclusion, we may assume without loss 
of generality that also in $f$ the total slack is 
$K - \alpha \lfloor K/\alpha \rfloor \leq \alpha - 1$. 
Set $\tilde K = \alpha \lfloor K/ \alpha \rfloor$. Then 
removing the slack in both $f_k$ and $f$, the remnants 
$\tilde f_k$ and $\tilde f$ now satisfy 
$\sum_s \tilde f(s) = \sum_s \tilde f_k(s) = \tilde K$.
The risks per crew remain unchanged, so $\tilde f$ 
is optimal. Moreover, all $\tilde f(s)$ and $\tilde f_k(s)$ are 
multiples of $\alpha$. Without loss of generality, we may 
now choose $\tilde f$ such that 
\[
\sum_{s\in \mathcal{S}} \left|
 \frac{\tilde f(s)}{\alpha} 
 - \frac{\tilde f_k(s)}{\alpha} 
\right|
\]
is smallest. 

The proof now proceeds along the same line as that of 
Theorem~\ref{t:cars}. If $\tilde f$ and $\tilde f_k$ are not 
the same, there exist $s$ and $s^\prime$ in $\mathcal{S}$
such that $\tilde f_k(s) > \tilde f(s)$ and 
$\tilde f_k(s^\prime) < \tilde f(s^\prime)$. 
We can therefore create a new solution $\hat f$ from 
$\tilde f$ by adding $\alpha$ resources to $s$ and 
withdrawing $\alpha$ units from $s^\prime$. Note that 
$\hat f(\cdot)$ satisfies the constraints and has
smaller $\ell_1$-distance to $\tilde f_k(\cdot)$.  

Observe that
\[
\frac{\Lambda(s^\prime)}
    { ( \tilde f(s^\prime)-\alpha) / \alpha }
    \leq \frac{\Lambda(s^\prime)}
     { \tilde f_k(s^\prime) / \alpha }
\mbox{ and }
   \frac{\Lambda(s)}
   { (\tilde f_k(s) -\alpha)/ \alpha}
   \leq  \frac{\Lambda(s)}
    {\tilde f(s) / \alpha }.
\]
Since the greedy algorithm has selected $\tilde f_k(s)$ but not
$\tilde f_k(s^\prime) + \alpha$,
\[
\frac{\Lambda(s^\prime)}
    { \tilde f_k(s^\prime) / \alpha  }
     \leq  \frac{\Lambda(s)}
    { (\tilde f_k(s) - \alpha)/ \alpha  }
\]
and therefore
\[
\frac{\Lambda(s^\prime)}
    {  ( \tilde f(s^\prime)-\alpha) / \alpha  }
    \leq \frac{\Lambda(s)}
    {  \tilde f(s) / \alpha  }.
\]
Hence 
\[
\frac{\Lambda(s^\prime)}
    { \lfloor \hat f(s^\prime) / \alpha \rfloor } =
\frac{\Lambda(s^\prime)}
    {  ( \tilde f(s^\prime)-\alpha) / \alpha } \leq 
    \frac{\Lambda(s)}
    {  \tilde f(s) / \alpha }
    \leq \max_{t\in \mathcal{S}} \frac{\Lambda(t)}
    { \lfloor \tilde f(t) / \alpha \rfloor }
\]
and
\[
\frac{\Lambda(s)}
    { \lfloor \hat f(s) / \alpha \rfloor } =
\frac{\Lambda(s)}
    {  ( \tilde f(s)+ \alpha) / \alpha} \leq 
    \frac{\Lambda(s)}
    {  \tilde f(s) / \alpha }
    \leq \max_{t\in \mathcal{S}} \frac{\Lambda(t)}
    { \lfloor \tilde f(t) / \alpha \rfloor },
\]
leading to a contradiction with the assumption that $\tilde f$
is an optimal allocation rule closest to $\tilde f_k$. The proof is complete.
\end{proof}

\section{Capacity planning for fire services in Twente}
\label{S:Twente}

In the security region Twente, The Netherlands, there are 
29 fire stations, the locations of which are plotted as the 
black dots in the left panel of Figure~\ref{F:risk2004}. 
For the catchment areas, we use the Voronoi cells. In other
words, station $s$ serves all locations in Twente that are 
closer to $s$ than to any other station. 

\begin{figure}[thb]
\begin{center}
\includegraphics[width=2.7in]{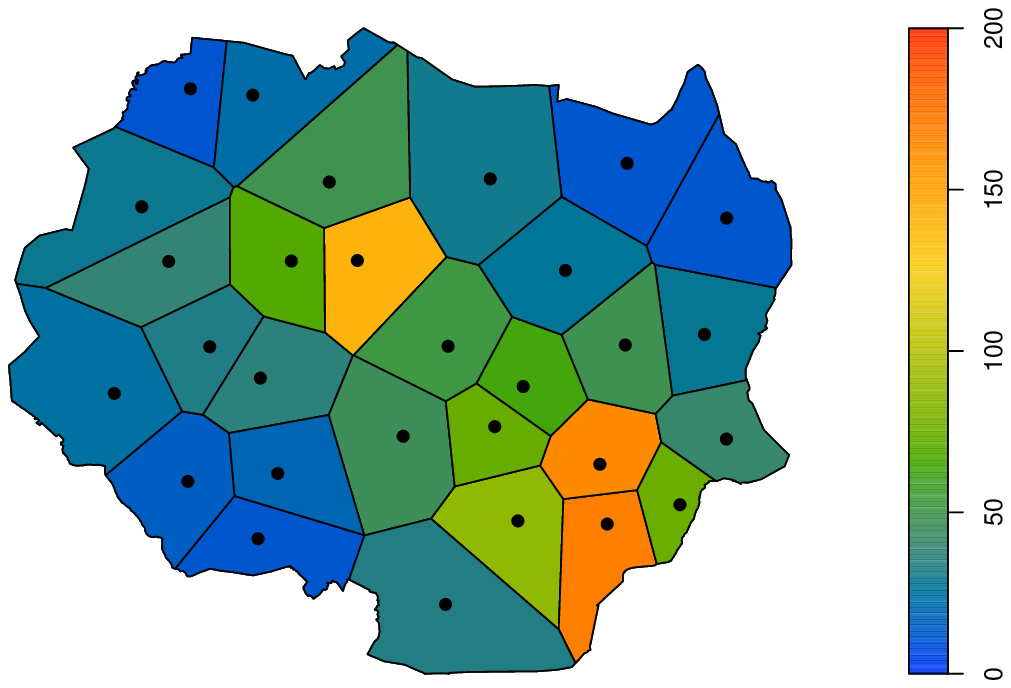}
\includegraphics[width=2.7in]{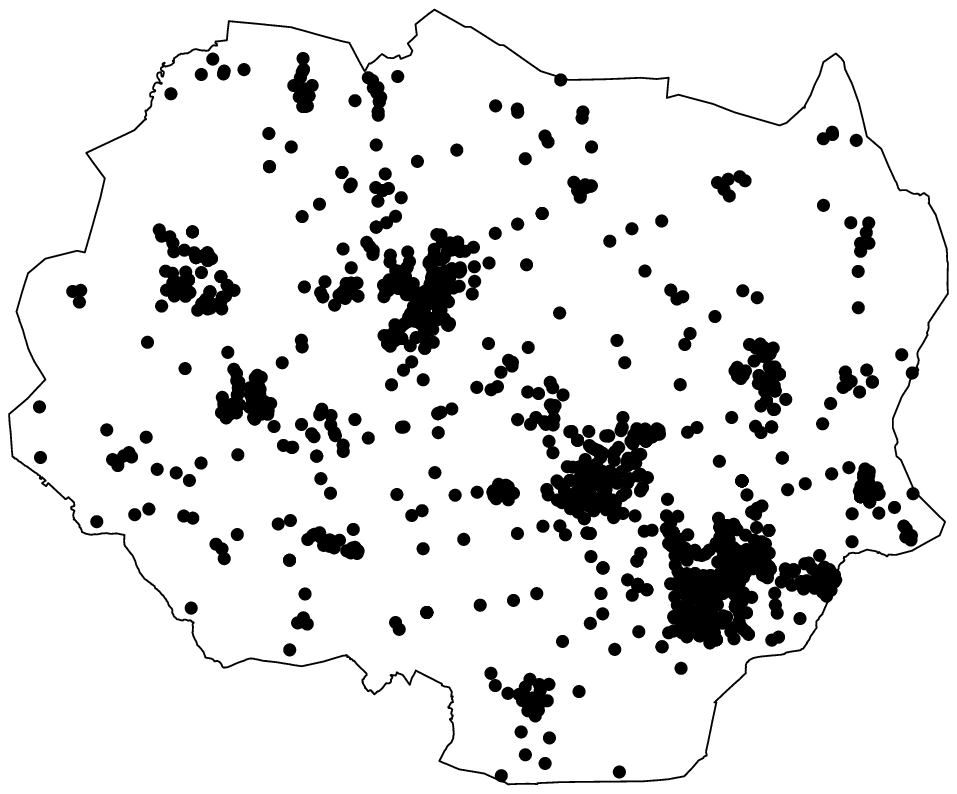}
\end{center}
\caption{Left: Voronoi cells around $29$ fire stations
in Twente coloured by risk. 
Right: mapped incidents that occurred in the year 2004.}
\label{F:risk2004}
\end{figure}

To estimate the $\Lambda(s)$, Twente Fire Brigade provided 
us with a map of incidents that occurred during the year 
2004 as shown in the right panel of Figure~\ref{F:risk2004}.
Note that the pattern is quite heterogeneous: many incidents
occur along a west-east diagonal that contains the three 
major cites in the region, there are fewer in the surrounding 
rural areas. Thus, we use an adaptive kernel estimator
(\ref{e:Abramson}), cf.\ Section~\ref{S:intensity}.
For the pilot estimate, the bandwidth is selected using 
cross-validation, that is, by maximising (\ref{e:bwPPL}).
The bandwidth in (\ref{e:Abramson}) is selected by minimising
(\ref{e:HamFun}). Finally, $\hat \lambda$ is integrated over 
the catchment areas to obtain the risks $\Lambda(s)$ shown
in Figure~\ref{F:risk2004} and tabulated for the most risky
regions in Table~\ref{tab:TwenteTable}. 


The last column in Table~\ref{tab:TwenteTable} lists the 
order in which Algorithm~\ref{A:pumper} assigns fire trucks
to stations up to $K=54$. Each stations receives its statutory 
single truck. After that, the first extra truck is assigned 
to the station serving the highest risk, $s_1$, located in the most 
densely populated city. The second truck goes to another 
station, $s_2$, in the same city, the third to the second city in 
Twente. After $K$ trucks have been allocated, the maximal risk 
per vehicle is $36.2$.

Let us truncate at $K= 39$ so that each station has at least 
one fire truck,  stations $s_1$, $s_2$ and $s_3$ have three, 
whilst $s_4$, $s_5$, $s_6$ and $s_7$ 
all have two trucks at their disposal. The next step is 
to allocate personnel. 
Now, a full crew consists of six people, so $\alpha = 6$. 
Table~\ref{tab:TwenteCrew} lists the index in which crew
members are assigned to stations by Algorithm~\ref{A:men}
for $K = 200$ after the 
$174$ members of the crews for each station have been 
allocated.
The first extra crews are assigned to $s_1$ and $s_2$ 
in the main city.
Stations $s_3$ and $s_4$
also receive an extra crew,
and two firefighters are assigned to $s_1$ to act as spares.
Observe that only $s_4$
is able to operate at full capacity. 

    \begin{table}[thb]
        \centering
        \begin{tabular}{|c|r|r|} \hline
        Station $s$ & $\Lambda(s)$ & Index allocated vehicles  \\
        \hline
        $s_1 $&  168.9 & 1, 5, 12, 20 \\
        $s_2$ &  163.2 & 2, 6, 22, 13 \\
        $s_3$ & 141.4 & 3, 9, 17 \\
        $s_4$ &  87.8 & 4, 19 \\
        $s_5$ & 75.3 & 7, 24 \\
        $s_6$ & 73.8 & 8, 25 \\
        $s_7$ & 68.6 & 10 \\
        $s_8$ & 63.9 & 11 \\
        $s_9$ & 52.8 & 14\\
        $s_{10}$ & 49.2 & 15\\
        $s_{11}$ & 48.5 & 16 \\ 
        $s_{12}$ & 46.2 & 18 \\
        $s_{13}$ & 41.4 & 21 \\
        $s_{14}$ & 38.4 & 23 \\
        \hline
        \end{tabular}
        \caption{Partial list of stations $s$ ordered according
        to risk $\Lambda(s)$. The last columns indicates the 
        order in which the greedy algorithm assigns vehicles to
        stations.}
        \label{tab:TwenteTable}
    \end{table}


   \begin{table}[hbt]
        \centering
        \begin{tabular}{|c|r|r|r|} \hline
        Station $s$ & $\Lambda(s)$ & n(s) &
           Index allocated firefighter  \\
        \hline
        $s_1$ &  168.9 &  3 & 1--6, 25, 26 \\
        $s_2$ &  163.2 &  3 & 7--12\\
        $s_3$ & 141.4 &  3 &  13--18 \\
        $s_4$ &  87.8 & 2 & 19--24 \\
        $s_5$ & 75.3 & 2 &      \\
        $s_6$ & 73.8 & 2 &  \\
        $s_7$ & 68.6 & 2 &  \\ \hline
       \end{tabular}
        \caption{Partial list of stations $s$ ordered according
        to risk $\Lambda(s)$. The third columns lists the 
        number, $n(s)$, of available fire trucks and the last columns 
        indicates the order in which the greedy algorithm assigns 
        firefighters to stations.}
        \label{tab:TwenteCrew}
    \end{table}

\section{Discussion}

In this paper, we combined tools from spatial statistics and
integer programming to solve resource allocation problems 
encountered in the security domain. We illustrated the approach
on data obtained from the Twente Fire Brigade. 

There are ample opportunities for further research.
So far, we assumed that vehicles and crew members are 
indistinguishable and can be allocated everywhere. 
In practice, this may not be the case. For example the Twente
Fire Brigade works with professionals as well as with volunteers.
The latter are usually tied to a specific station, the former
may have different roles and expertise.  Also, there are 
different types of vehicles, for instance trucks equipped 
with high ladders or trucks specifically designed for wildfires. 
Clearly, the countryside has different needs than an urban 
environment. Finally, there may be scheduling restrictions due 
to training and legal obligations.

\bigskip 

\noindent
{\bf Acknowledgements}: Thanks to P.\ Visscher
and R.\ de Wit from Twente Fire Brigade for access to data and 
many helpful discussions.

\end{document}